\documentclass[a4,12pt,reqno,fleqn]{amsart}
\usepackage{graphicx}
\usepackage{amscd}
\usepackage{xfrac}
\usepackage{mathtools}
\emergencystretch=\maxdimen
\allowdisplaybreaks[1]
\usepackage{array}
\usepackage{xcolor}
\usepackage[all]{xy}
\makeatletter
\@namedef{subjclassname@2020}{\textup{2020} Mathematics Subject Classification}
\makeatother


\usepackage{amsmath}
\usepackage{amsthm}
\usepackage{amsfonts,amssymb,hyperref,mathrsfs,cleveref,setspace,enumitem,verbatim}
\usepackage[margin=3.5cm]{geometry}

\DeclareMathAlphabet{\mathpzc}{OT1}{pzc}{m}{it}
\DeclareMathOperator{\sign}{sign}
\newtheorem{thm}{Theorem}[section]

\newtheorem{lem}[thm]{Lemma}
\newtheorem{propn}[thm]{Proposition}
\theoremstyle{definition}

\newtheorem{remark}[thm]{Remark}
\newtheorem{example}[thm]{Example}
\newcommand{\omin}{\otimes^{\min}}
\newcommand{\bj}{\perp_{BJ}}
\newcommand{\nbj}{\not\perp_{BJ}}
\newcommand{\olm}{\otimes^{\lambda}}
\newcommand{\mC}{\mathbb{C}}

\newcommand{\ot}{\otimes}

\author[Mohit and R.~Jain]{Mohit and Ranjana Jain}

\address{Mohit, Department of  Mathematics, University of Delhi, Delhi}
\email{mohitdhandamaths@gmail.com}

\address{Ranjana Jain, Department of Mathematics, University of Delhi, Delhi}
\email{rjain@maths.du.ac.in}

\begin{document}
\title[Orthogonality in certain tensor products of Banach spaces]{Birkhoff-James orthogonality in certain tensor products of Banach spaces}
\maketitle
\textbf{Abstract:} In this article, the relationship between Birkhoff-James orthogonality of elementary tensors in certain tensor product spaces with the Birkhoff-James orthogonality of individual elements in their respective spaces is studied.

\section{Introduction}

Orthogonality plays a vital role in the study of  Hilbert spaces. In the last few decades, there have been several generalizations of the notion of orthogonality in the setting of  normed spaces. Among them, Birkhoff-James orthogonality (\cite{GB,H11}) has been studied extensively and has found various significant relationships with other geometric properties of normed spaces.
 For a normed space $X$ over the scalar field $\mathbb{F}$ ($\mathbb{R}$ or $\mathbb{C}$)  and $x, y\in\;X,$ we say that $x$ is {\it Birkhoff-James orthogonal} to $y$ (denoted by $x\perp_{BJ}y$) if $$||x+\alpha y||\geq||x||,\  \text{for all} \ \alpha\in\;\mathbb{F}.$$
 
 It is closely related to the notion of best approximation to a point in a subspace as is evident from the fact that for a normed space $V$ and a subspace $W$ of $V$, a point $w \in W$ is best approximation to $v$ in $V$ if and only if $v-w \bj W$. It is also useful to obtain some distance formulae in normed spaces.
  For more details on the geometric significance of BJ- orthogonality, one may refer to \cite{G2,GB,H13,H10,H11} and the references therein.
  
On the other hand, as has been observed with various categories, the theories of tensor products of Banach spaces  and of $C^*$-algebras are indispensable to the understanding of these objects. Thus, it is quite natural to analyze the notion of Birkhoff-James orthogonality in tensor product spaces.  Our interest in fact grew from this very basic curiosity. Quite surprisingly,  this aspect has been hardly touched upon and there are no substantial results in the literature.  Closely related to our interest, Light and Cheney (in \cite{cheney}) discussed the theory of approximation in terms of proximinality in various tensor products of spaces which can be linked with the Birkhoff-James orthogonality to some extent.  For instance, they proved that for subspaces $G$ and $H$ having linear proximity maps in Banach spaces $X$ and $Y$ respectively,  the subspace $\overline{X {\otimes} H}^{\alpha} +\overline {G{\otimes}Y}^{\alpha}$ also has a linear proximity map, and is proximinal in $X\otimes_{\alpha}Y$, where $\otimes_{\alpha}$ is a uniform cross norm. In particular,  for finite-dimensional subspaces $G$ and $H$ of $L_{1}(S)$ and $L_{1}(T)$, respectively, $S$ and $T$ being finite measure spaces, the subspace $L_{1}\otimes H+G\otimes L_{1}(T)$ is proximinal in $L_{1}(S\times T)$.
	 
It is natural to ask  how Birkhoff-James orthogonality of elementary tensors in various tensor products of Banach spaces is related to the Birkhoff-James orthogonality of the individual elements in their respective spaces. As a first step in this direction, for Banach spaces $X$ and $Y$, any cross norm $\|\cdot\|_{\alpha}$ on their algebraic tensor product $X\otimes Y$ and elements $x_1,x_2 \in X$, $y_1,y_2 \in Y$, one can ask the following questions:
\begin{enumerate}
	\item If $x_{1}\perp_{BJ}x_{2}$ and $y_{1}\perp_{BJ}y_{2} $ then is $x_{1}\otimes y_{1}\perp_{BJ}x_{2}\otimes y_{2}$?
	\item If $x_{1}\otimes y_{1}\perp_{BJ}x_{2}\otimes y_{2}$ in $X\otimes^{\alpha}Y$ then is $x_{1}\perp_{BJ}x_{2}$ and(or) $y_{1}\perp_{BJ}y_{2} $?
	\item For subspaces $X_1$ and $Y_1$  of $X$ and $Y$ respectively, and $x\ot y \bj X_1 \ot Y_1$, can we say that $x \bj X_1$ or $y \bj Y_1$?
\end{enumerate}

 In this article we make an attempt to answer these questions by considering different cross norms on $X\otimes Y$. We show that the first assertion is always true, whereas we provide an example (see Example \ref{failexample}) to show that Question 2  does not have an affirmative answer in general. This also leads to a negative answer of Question 3 (see, \ref{bj-subspaces}). More precisely, we demonstrate that it fails in $C_{\mathbb{C}}(K_{1})\otimes^{\alpha}C_{\mathbb{C}}(K_{2})$, where $C_{\mathbb{F}}(K)$ denotes the space of $\mathbb{F}$-valued continuous functions defined on a compact Hausdorff space $K$  equipped with the sup norm and $\|\cdot\|_{\alpha}$ is any cross norm. However, this is not a general phenomenon. By deploying different techniques, we discuss various spaces and tensor products for which Question 2 in fact has a positive answer. In particular, we prove that it has an affirmative answer for the following spaces:
 \begin{itemize}
 	\item  $X \otimes^{\lambda} Y$, where $X$ and $Y$ are real normed spaces and $\otimes^{\lambda}$ is the injective tensor product;
 	\item $L^{p}(S,\mu)\otimes^{\Delta_{p}} L^{p}(T,\nu),$ where $1< p<\infty$, $(S, \mu)$ and $(T, \nu)$ are positive measure spaces and $\otimes^{\Delta_{p}}$ is the natural norm on the tensor product of $p$-integrable functions (see (\ref{pnorm});
 	\item $B(H)\otimes^{\min}B(K)$ (for elementary tensors of rank one operators), $H$ and $K$ being Hilbert spaces;
 	\item $C_{\mathbb{C}}(K_{1})\otimes^{\lambda}C_{\mathbb{C}}(K_{2})$, with Birkhoff-James orthogonality being replaced by strong Birkhoff-James orthogonality.
 	\end{itemize}

\section{Preliminaries}
Let us first recall some definitions and different notions of tensor norms. Given Banach spaces $X$ and $Y,$ {\it Banach space injective tensor norm} on the algebraic tensor product $X\otimes Y$ is defined as
 $$||u||_{\lambda}= \sup \left\{ \left|\sum\limits_{i=1}^n{f(x_{i})g(y_{i})}\right|: f\in\;B_{X^{*}}, g\in\;B_{Y^{*}}\right\}, u=\sum\limits_{i=1}^nx_{i}\otimes y_{i} \in X\otimes Y,$$
 where $B_{X^*}$ denotes the closed unit ball of $X^*$. The completion of $X\otimes Y$ with respect to $\|\cdot \|_{\lambda} $ is  called the {\it Banach space injective tensor product} of $X$ and $Y$ and is denoted by $X\otimes^{\lambda}Y$. Similarly, given two $C^{*}$-algebras $A$ and $B$, the minimal tensor norm on the algebraic tensor product $A\otimes B$ is defined as$$||u||_{\min}=\sup\left\{||(\pi_{1}\otimes\pi_{2})(u)||\right\},u\in A\otimes B,$$ where $\pi_{1},\pi_{2}$ denote the representations of $A$ and $B$ respectively. The completion of $A\otimes B$ with respect to $\|\cdot\|_{\min}$ is called the minimal $C^*$-tensor product of $A$ and $B$ and is denoted by $A\otimes^{\min}B.$ For a detailed study of these norms one may refer to \cite{cheney,H3,H2}. A norm $\| \cdot \|_{\alpha} $ on $X\otimes Y$ is said to be  a {\it cross norm} if $\| x\otimes y\|_{\alpha} = \|x\| \|y\|$ for all $x\in X, y \in Y$.
 
  \vspace*{1mm}
  
  We now recall and collect some important results and identifications which will be useful in the later discussion. The first in the list is a basic yet an important and useful characterization of BJ-orthogonality in terms of the linear functionals. 
  
  \begin{thm}\cite[Theorem 2.1]{H11}\label{james}
  	In a normed space $X$, $x\perp_{BJ} y $ if and only if there exists a linear functional $f \in X^*$ such that $|f(x)|= \| f\| \|x\|$ and $f(y)=0$.
  \end{thm}
   It is a well known fact that the tensor product of continuous functions over compact spaces can be treated as a continuous function on an appropriate compact space. More precisely,
  
  \begin{thm}\cite[4.2 (3)]{H6}\label{idef1}
  	Let $K_{1}$ $ K_{2}$ be compact Hausdorff spaces. The injective tensor product $C_{\mathbb{F}}(K_{1})\otimes^{\lambda}C_{\mathbb{F}}(K_{2})$ is isometrically isomorphic to the space $C_{\mathbb{F}}(K_{1}\times K_{2})$ via the identification $f_{1}\otimes f_{2}\rightarrow f_1f_2$ where $f_1f_2: K_{1}\times K_{2}\rightarrow \mathbb{F}$ is defined as $f_1f_2(k_{1},k_{2})=f_{1}(k_{1})f_{2}(k_{2})$ for $k_i \in K_i$.	
  \end{thm}
  
  There is an interesting geometric charaterization of BJ-orthogonality in $C_{\mathbb{C}}(K)$ which follows immediately from \cite[Corollary 2.1]{H8}.
  \begin{thm}\label{C(K)}
  	Let $K$ be a compact Hausdorff space, $f,g \in\;C_{\mathbb{C}}(K)$ and $M_{f}:=\{t\in\;K : |f(t)|=||f||\}.$ Then $f\perp_{BJ}g$ if and only if the set $A=\{\overline{f(t)}g(t) : t\in\;M_{f}\}$ is not contained in an open half plane (in $\mathbb{C}$) with boundary that contains the origin. Equivalently, $f\perp_{BJ}g$ if and only if the closed convex hull of $A$ contains the origin. 
  \end{thm}

We have an elegant characterization of BJ-orthogonality for vector valued continuous functions also.

\begin{thm}\cite[Theorem 2.1]{RSS}\label{C(K,X)}  
Let $K$ be a compact Hausdorff space, $X$ be a real normed space and $f,g \in C_{\mathbb{R}}(K,X)$ be non-zero elements. Then $f \perp_{BJ} g$ if and only if there exists $k_1,k_2 \in M_f$ such that $g(k_1)	\in f(k_1)^+$ and $g(k_2)	\in f(k_2)^-$,
	\end{thm}

On the similar lines we have some identification for the $p$-integrable functions. 


For an arbitrary measure space $(S,\mu)$ and a Banach space $X$, let  $L^{p}(\mu, X)$, $1\leq p<\infty$, denote the space of Bochner $p$-integrable functions from $S$ to $X.$ Consider the norm on $L^{p}(S, \mu)\otimes X$, written as $\Delta_{p}$, induced via the injective map $\phi:L^{p}(S,\mu)\otimes X\rightarrow L^{p}(\mu, X)$ given by $\phi(f\otimes x)=f(\cdot)x$, that is, 
\begin{equation}\label{pnorm}
	||u||_{\Delta_{p}}=||\phi(u)||_{L^{p}(\mu, X)}, \quad \forall \ \ u\in L^{p}(S, \mu)\otimes X.
\end{equation}
We take $X=L^{p}(T, \nu)$ and assume that $L^{p}(S,\mu)\otimes^{\Delta_{p}} L^{p}(T,\nu)$ denotes the completion of  $L^{p}(S,\mu)\otimes L^{p}(T,\nu)$ with respect to the $\Delta_{p}$-norm. 
\begin{thm}\cite[Section 7.2]{H6}\label{identif2}	
	The space $ L^{p}(S, \mu) \otimes^{\Delta_{p}} L^{p}(T, \nu) $ is isometrically isomorphic to $L^{p}(S\times T, \mu\times\nu)$, $1<p<\infty$. 
	
\end{thm}

Birkhoff-James orthogonality for the $p$-integrable functions can be derived from \cite[Theorem 1.11]{H5} in the following form:
\begin{thm}\label{identif}
	Let $(S, \mu)$ be a  positive measure space and $f, g\in L^{p}(S, \mu)$, $1 < p<\infty$ such that $f \not \in span\{g\}.$ Then $f\perp_{BJ}g$ if and only if 		
$$\int\limits_{S}{g(s)|f(s)|^{p-1}\sign(f(s))\,d\mu(s)}=0.$$

\end{thm}

  \section{Main Results}
  
  We start with a proof of our first  assertion.
\begin{thm}\label{thm:2.1}
Let $X$ and $Y$ be Banach spaces, $x_{1},x_{2}\in X$ with $x_{1}\perp_{BJ}x_{2}.$ Then for any $y_{1},y_{2}\in Y$, $ x_{1}\otimes y_{1}\perp_{BJ}x_{2}\otimes y_{2}$ in $X\otimes^{\alpha}Y$, where $\alpha$ is any cross norm on $X\otimes Y.$
\end{thm}
\begin{proof}
 Since $x_{1}\perp_{BJ}x_{2}$, by Theorem \ref{james}, there exists a $\phi_{0}\in X^{*}$ with $||\phi_{0}||=1$ such that $|\phi_{0}(x_{1})|=||x_{1}||$ and $\phi_{0}(x_{2})=0.$  Consider $\psi_{0}\in Y^{*}$ such that $||\psi_{0}||=1$ and $\psi_{0}(y_{1})=||y_{1}||.$  Now for any $\mu\in \mathbb{C},$ we have 

\begin{tabular}{ccl}
$\|(x_{1}\otimes y_{1})+\mu (x_{2}\otimes y_{2})||_{\alpha} $ & $\geq $ & $\|(x_{1}\otimes y_{1})+\mu (x_{2}\otimes y_{2})\|_{\lambda} $\\ 
 & = & $\sup\{|\phi(x_{1})\psi(y_{1})+\mu\phi(x_{2})\psi(y_{2})|: \phi\in B_{X^{*}} , \psi\in B_{Y^{*}}\} $\\
 & $\geq$ & $ | \phi_{0}(x_{1})\psi_{0}(y_{1})+\mu\phi_{0}(x_{2}) \psi_{0}(y_{2})|$ \\
& = & $\|x_{1}\| \|y_{1}\|=\|x_{1}\otimes y_{1}\|_{\alpha}$
\end{tabular}

This completes the proof. 

\end{proof}

In general, converse of the above result is not true.
\begin{example}\label{failexample}

	Let $K_{1}=\{x_{1},x_{2},x_{3}\}$ and $K_{2}=\{y_{1},y_{2},y_{3}\}$ be compact Hausdorff spaces equipped with discrete topology. Consider $f_{1},f_{2}\in C_{\mathbb{C}}(K_1)$ defined as
	$f_1(x_1)=1, f_1(x_2)=1+ 2i, f_1(x_3)=1-2i$ and $f_{2}(x_i)=1$ for $i=1,2,3$. Also, consider $g_{1},g_{2} \in C_{\mathbb{C}}(K_2)$ defined as $g_{1}(y_1)= 1, g_1(y_2)= -1+2i, g_1(y_3)= -1-2i$ and $g_{2}(y_i)=2$ for all $i=1,2,3$. Note that $\|f_1\| = \|g_1\| = \sqrt{5}$,  $M_{f_1} = \{ x_2, x_3\}$ and $M_{g_1} = \{ y_2, y_3\}$.
	Thus, closed convex hull of the sets $\{\overline{f_{1}(t)}f_{2}(t) : t\in\;M_{f_{1}}\}$ and $\{\overline{g_{1}(s)}g_{2}(s) : s\in\;M_{g_{1}}\}$ are the line joining $1+2i,1-2i$ and the line joining $-2+4i,-2-4i$, respectively. Since both the closed convex hulls do not contain the origin, by Theorem \ref{C(K)}, neither $f_{1}\perp_{BJ}f_{2}$ nor $g_{1}\perp_{BJ}g_{2}.$ On the other hand,  $\|f_1 \otimes g_1\| = \|f_1\| \|g_1\| = 5 $ and $M_{f_{1}g_{1}}=\{(x_{2},y_{2}), (x_{2}, y_{3}), (x_{3}, y_{2}), (x_{3}, y_{3})\}$. Since the closed convex hull of the set $\{\overline{f_{1}(t)g_{1}(s)}f_{2}(t)g_{2}(s) : (t,s)\in\;M_{f_{1}g_{1}}\}$ is the triangle with vertices $-10, 6+8i,6-8i$ and it contains the origin, thus, again by  Theorem \ref{C(K)}, $f_{1}g_{1}\perp_{BJ}f_{2}g_{2}$ in $C_{\mathbb{C}}(K_{1}\times K_{2})$. Equivalently, $f_{1}\otimes g_{1}\perp_{BJ}f_{2}\otimes g_{2}$ in $C_{\mathbb{C}}(K_{1})\otimes^{\lambda}C_{\mathbb{C}}(K_{2})$. Since injective norm $\| \cdot \|_{\lambda}$  is the least cross norm on $X\otimes Y$, it follows that  $f_{1}\otimes g_{1}\perp_{BJ}f_{2}\otimes g_{2}$ in $C_{\mathbb{C}}(K_{1})\otimes^{\alpha}C_{\mathbb{C}}(K_{2})$, for any cross norm $\| \cdot \|_{\alpha}$.
\end{example}
\begin{remark}\label{bj-subspaces}
 By considering the subspaces $X_2$ and $Y_2$  spanned by $f_{2}$ and $g_{2}$ respectively in the above example, we observe that   $f_1\ot g_1 \bj X_2 \ot Y_2$, but neither $f_1 \bj X_2$ nor $g_1 \bj Y_2$. 
\end{remark}
It is also interesting to note that $x_1 \otimes x_2 \bj y_1 \otimes y_2$ does not necessarily imply $x_1 \bj y_1$ and $y_1 \bj y_2$ both. This can be illustrated by the following (and various other) examples:
\begin{example}
For the identity function $Id \in C_{\mathbb{C}}[0,1]$ and the constant unity function $1$, $Id\otimes 1\bj 1\otimes Id$ in $C_{\mathbb{C}}[0,1]\olm C_{\mathbb{C}}[0,1]$. This can be seen using Theorem \ref{idef1}, for any $\mu\in\mathbb{C}$, we have
$$||Id \ot 1+\mu 1 \ot Id||_{\lambda}=||Id.1+\mu 1.Id||=\underset{x,y\in [0,1]}{\sup}|x+\mu y|\geq 1=||Id \otimes 1||.$$ However, $Id\nbj 1$ as $||Id-(1/2)1||=1/2\ngeq ||Id||=1$. The fact that $ 1 \bj Id$ follows easily from the definition.
\end{example}
\begin{example}
Consider the matrices $A,B,C \in M_{2}(\mC)$ given by $A=$
$\begin{bmatrix}
1 & 0\\0 & 2
\end{bmatrix}$, $B=$
$\begin{bmatrix}
1 & 0\\0 & 0
\end{bmatrix}$, $C=$
$\begin{bmatrix}
1 & 1\\1 & 0
\end{bmatrix}$.	We claim that $A\otimes B\bj C\otimes I$ in $M_2({\mC})\omin M_2({\mC}).$ Since $M_2({\mC})\omin M_2({\mC})$ can be identified with $M_{4}({\mC})$ equipped with the spectral norm. Therefore, it is equivalent to check that $P \bj Q$ where $P=$
$\begin{bmatrix}
1 & 0 & 0&0\\0&0&0&0\\0&0&2&0\\0 & 0 & 0 & 0
\end{bmatrix}$ and $Q=$
$\begin{bmatrix}
1&0&1&0\\0&1&0&1\\1&0&0&0\\0&1&0&0
\end{bmatrix}$.
It is easy to check that for the unit vector $x=[
0 \ 0 \ 1 \ 0]^t$ we have $||Px||=2=||P||,\;\langle Px,Qx \rangle=0.$ Thus, by a well known characterization given by Bhatia and Semrl \cite[Theorem 1.1]{G2}, $P\bj Q$. Here $A \bj C$ as for $x=(0,1)^t, \|Ax\|=2 = \|A\|$ and $\langle Ax,Cx \rangle =0$. However, $B\nbj I$ as for any unit vector $x=(x_1,x_2)^t$ in $\mathbb{C}^2$,  $||Bx||=1=||B||$ will imply that $|x_1| =1$ but for such an $x$, $\langle Bx,x \rangle\neq 0.$ 
\end{example}
Here is another non-trivial example.
\begin{example}
	Consider the operators $P,Q,R \in B(\ell^2)$ given by 
	$ P(x_{1},x_{2},.....)= (x_{1}, 0,0,.....); $
	$ Q(x_{1},x_{2},.....)= (x_{1}, x_{2},0,.....)$ and $ R(x_{1},x_{2},.....)= (0, x_{1},x_{2},.....)$, and let $I\in B(\ell^2)$ be the identity operator. We claim that $P\otimes I\perp_{BJ}Q\otimes R$ in $B(\ell^2) \omin B( \ell^2)$. To see this, we use the inclusion $B(\ell^2) \omin B( \ell^2) \subseteq B(\ell^2 \bar{\otimes} \ell^2)$, where $\bar{\otimes}$ represents the Hilbertian tensor product, and consider $\lambda \in \mathbb{C},$ 
	\begin{align*} ||(P\otimes I)+\lambda(Q\otimes R)||^{2}&=\underset{h\in \ell^2\bar{\otimes} \ell^2,\;||h||=1}{\sup}||(P\otimes I)h+\lambda(Q\otimes R)h||^{2}\\
		&\geq||(P\otimes I)(e_{1}\otimes e_{2})+\lambda(Q\otimes R)(e_{1}\otimes e_{2})||^2\\
		&=||e_{1}\otimes e_{2}+\lambda e_{1}\otimes e_{3}||^2\\
		&=\langle e_{1}\otimes e_{2}+\lambda e_{1}\otimes e_{3},e_{1}\otimes e_{2}+\lambda e_{1}\otimes e_{3}\rangle\\
		&=1+|\lambda|^2 \geq 1 = ||P\otimes I||
	\end{align*}
	
	Here $I \bj R$ as for the constant sequence of unit vectors $\{h_{n}\}  \in \ell^2$ where $h_n = e_1$ for all $n$, we have $\lim_{n \to\infty}||I(h_{n})||=|| e_1|| =1=||I||$ and $\lim_{n \to\infty} \langle I(h_{n}),R(h_{n}) \rangle= \langle e_1,e_2 \rangle = 0$ and thus orthogonality follows from \cite[Remark 3.1]{G2}. However, $P\not\perp_{BJ}Q$, if not, then we get a sequence of unit vectors say $\{h_{n} = (h_n^{(m)})\}$ in $\ell^2$ such that $\lim_{n \to\infty}||P(h_{n})||=\lim_{n \to\infty}|h^{(1)}_{n}| =||P||={1}$ and $\lim_{n \to\infty} \langle P(h_{n}),Q(h_{n}) \rangle=\lim_{n \to\infty} |h^{(1)}_{n}|^{2}= 0$, which is not possible. 
\end{example}

In many spaces the converse of Theorem \ref{thm:2.1} holds true. We shall now discuss few of such spaces. For a real normed space $X$ and $0 \neq x \in X$, we use the notations, as given in \cite{H13}, $x^+ = \{y\in X : \|x+ \lambda y\| \geq \|x\| \ \forall \lambda \geq 0 \}$ and  $x^- = \{y\in X : \|x+ \lambda y\| \geq \|x\| \ \forall \lambda \leq 0 \}$. It is known that for any $x,y \in X$, either $y\in x^+$ or $y \in x^-$, see \cite[Proposition 2.1]{H13}.  

\begin{thm}
	Let $X$ and $Y$ be real Banach spaces and $x_{1},x_{2}\in X, y_{1},y_{2}\in Y$. Then $x_{1}\otimes y_{1}\perp_{BJ}x_{2}\otimes y_{2}$ in $X\otimes^{\lambda}Y$ if and only if either $x_{1}\perp_{BJ}x_{2}$ or $y_{1}\perp_{BJ}y_{2}.$
\end{thm}
\begin{proof} In view of Theorem \ref{thm:2.1}, we only need to prove the necessary part. Suppose that $v_{1}=x_{1}\otimes y_{1}\perp_{BJ}x_{2}\otimes y_{2}=v_{2}$ in $X\otimes^{\lambda}Y.$ By \cite[Section 3.4]{H3}, we know that  $X\otimes^{\lambda}Y$ can be embedded (isometrically) as a subspace of $C_{\mathbb{R}}(B_{X^{*}}\times B_{Y^{*}})$ via the embedding $(\sum_{i=1}^n x_i \ot y_i) \hookrightarrow \sum_{i=1}^n f_{x_i} f_{y_i}$, where $f_{x}:X^{*} \to \mathbb{R}$ represents the map $f_{x}(\phi)=\phi(x),$ for all $\phi \in X^*$;    
	 here  $B_{X^{*}},B_{Y^{*}}$ are closed unit balls in $X^{*}$ and $Y^{*}$ respectively equipped with the weak$^*$-topology. Note that, by Banach Alaoglu's Theorem,   $B_{X^{*}}\times B_{Y^{*}}$ is compact, $B_{X^{*}}$ and $B_{Y^{*}}$ being compact.
	Thus, $v_{1}\perp_{BJ}v_{2}$ in $X\otimes^{\lambda}Y$ is equivalent to saying that $f_{x_{1}}f_{y_{1}}\perp_{BJ}f_{x_{2}}f_{y_{2}}$ in $C_\mathbb{R}(B_{X^{*}}\times B_{Y^{*}}).$
	We claim that either $f_{x_{1}}\perp_{BJ}f_{x_{2}}$ in $C_\mathbb{R}(B_{X^{*}})$ or $f_{y_{1}}\perp_{BJ}f_{y_{2}}$ in $C_\mathbb{R}(B_{Y^{*}})$. Let, if possible, neither $f_{x_{1}}\perp_{BJ}f_{x_{2}}$ nor $f_{y_{1}}\perp_{BJ}f_{y_{2}}.$ Then, by Theorem \ref{C(K,X)},  either $f_{x_{2}}(\phi)\in f_{x_{1}}(\phi)^{+}$ or $f_{x_{2}}(\phi)\in f_{x_{1}}(\phi)^{-}\;\forall \;\phi \in M_{f_{x_{1}}}$, and similarly, either $f_{y_{2}}(\psi)\in f_{y_{1}}(\psi)^{+}$ or $f_{y_{2}}(\psi)\in f_{y_{1}}(\psi)^{-}\;\forall \;\psi \in M_{f_{y_{1}}}.$ Let us discuss the case when $f_{x_{2}}(\phi)\in f_{x_{1}}(\phi)^{+}$ and $f_{y_{2}}(\psi)\in f_{y_{1}}(\psi)^{+}$ for all $\phi \in M_{f_{x_{1}}}, \psi \in M_{f_{y_{1}}}$. Then, we have $f_{x_{1}}(\phi)f_{x_{2}}(\phi)>0$ and $f_{y_{1}}(\psi)f_{y_{2}}(\psi)>0$  for all $\phi \in M_{f_{x_{1}}}, \psi \in M_{f_{y_{1}}}$. This gives  $f_{x_{1}}(\phi)f_{x_{2}}(\phi)f_{y_{1}}(\psi)f_{y_{2}}(\psi)>0\;\forall\;(\phi,\psi)\;\in M_{f_{x_{1}}}\times M_{f_{y_{1}}}=M_{f_{x_{1}}f_{y_{1}}}$ which contradicts the hypothesis $f_{x_{1}}f_{y_{1}}\perp_{BJ}f_{x_{2}}f_{y_{2}}.$ So this case is not possible. Similarly, we will get contradiction in rest of the three cases.
	
	Now,  assume that $f_{x_{1}}\perp_{BJ}f_{x_{2}}.$  Therefore for any scalar $\mu\in \mathbb{R},$ we have  $\sup\{|f_{x_{1}}(\phi)+\mu f_{x_{2}}(\phi)|: \phi\in B_{X^{*}}\}\geq \sup\{|f_{x_{1}}(\phi)|: \phi\in B_{X^{*}}\}$ which, using a consequence of Hahn Banach Theorem, gives that $||x_{1}+\mu x_{2}||\geq ||x_{1}||$ and thus $x_1 \bj x_2$. Similarly  $f_{y_{1}}\perp_{BJ}f_{y_{2}}$ would imply that $y_1 \bj y_2$. 
\end{proof}
Birkhoff-James orthogonality behaves well for the injective tensor norm on the real Banach spaces. Moving from the smallest cross norm to the largest cross norm on the tensor product of Banach spaces, we are now interested in analysing the spaces for which the BJ-orthogonality has a sound behaviour with respect to the projective tensor norm. We have already seen in Example \ref{failexample} that the   space of continuous complex valued functions is not the right candidate. We next move to the space of $p$-integrable functions.

 %
\begin{thm}
	Let $(S, \mu)$ and $(T, \nu)$ be positive measure spaces. Then $f_{1}\otimes f_{2}\perp_{BJ}g_{1}\otimes g_{2}$ in $L^{p}(S, \mu)\otimes^{\Delta_{p}}L^{p}(T, \nu)$, $1 < p<\infty$ if and only if either $f_{1}\perp_{BJ}g_{1}$ or $f_{2}\perp_{BJ}g_{2}.$
\end{thm}
\begin{proof}
	Using the identification mentioned in Theorem \ref{identif2}, we know that $f_{1}\otimes f_{2}\perp_{BJ}g_{1}\otimes g_{2}$ in $L^{p}(S, \mu)\otimes^{\Delta_{p}}L^{p}(T, \nu)$ if and only if $f_{1}f_{2}\perp_{BJ}g_{1}g_{2}$ in $L^{p}(S\times T, \mu\times\nu).$ Let if possible, neither $f_{1}\not \perp_{BJ}g_{1}$ nor $f_{2}\not \perp_{BJ}g_{2}.$
	 By Theorem \ref{identif}, we have $$\int\limits_{S}{g_{1}(s)|f_{1}(s)|^{p-1}\sign(f_{1}(s))\,d\mu(s)}\neq0 \quad \text{and} $$ $$\int\limits_{T}{g_{2}(t)|f_{2}(t)|^{p-1}\sign(f_{2}(t))\,d\nu(t)}\neq0.$$ 
	This implies that  $$\left(\int\limits_{S}{g_{1}(s)|f_{1}(s)|^{p-1}\sign(f_{1}(s))\,d\mu(s)}\right) \left(\int\limits_{T}{g_{2}(t)|f_{2}(t)|^{p-1}\sign(f_{2}(t))\,d\nu(t)}\right) $$
	$$ =\iint\limits_{S\times T}{g_{1}(s)g_{2}(t)|f_{1}(s)f_{2}(t)|^{p-1}\sign(f_{1}(s)f_{2}(t))\,(d\mu(s)\times d\nu(t))}\neq0,$$ which implies that $f_{1}f_{2}\not \perp_{BJ}g_{1}g_{2}$ in $L^{p}(S\times T,\mu\times\nu)$. Hence the proof.
\end{proof}


We now turn our attention back  to the space of continuous complex valued functions. It was illustrated in Example \ref{failexample} that the BJ-orthogonality doesn't work well for the tensor product $C_\mathbb{C}(K_{1})\otimes^{\alpha}C_\mathbb{C}(K_{2})$ with respect to any cross norm $\alpha$. Although in a special situation it does work well.

\begin{propn}
	Let $f_{1} \in\;C_\mathbb{C}(K_{1})$ and $g_{1} \in\;C_\mathbb{C}(K_{2})$ both attain their norms at a unique point. Then, for any $f_{2}\in\;C_\mathbb{C}(K_{1})$ and $g_{2}\in\;C_\mathbb{C}(K_{2})$,  $f_{1}\otimes g_{1}\perp_{BJ}f_{2}\otimes g_{2}$ in $C_\mathbb{C}(K_{1})\otimes^{\lambda}C_\mathbb{C}(K_{2})$ implies $f_{1}\perp_{BJ} f_{2}$ or $g_{1}\perp_{BJ} g_{2}.$
	
\end{propn}
\begin{proof}
	It follows easily from Theorem \ref{C(K)}.
\end{proof}

However, if we treat  $C_\mathbb{C}(K)$, $K$ being a compact Hausdorff space, as a $C^*$-algebra then we have the desired result with BJ-orthogonality being replaced by the strong BJ-orthogonality. 
Recall that, if $H$ is a right Hilbert $C^{*}$-module over a $C^{*}$-algebra $A$, then for $h_1, h_2 \in H$, we say that $h_{1}$ is strongly Birkhoff-James orthogonal to $h_{2}$ (written as $h_{1}\perp_{BJ}^{s}h_{2}),$ if $||h_{1}+h_{2}a||\geq ||h_{1}||$ for all $a\in A$, see \cite{H9}. It is obvious that any $C^*$-algebra is a Hilbert $C^{*}$-module over itself. Also, it is worth mentioning that for commutative $C^*$-algebras, the injective norm coincides with the minimal C*-tensor norm \cite[Theorem 4.14]{H2}. More precisely, $C_\mathbb{C}(K_{1})\otimes^{\lambda}C_\mathbb{C}(K_{2})$ is isometrically isomorphic to $C_\mathbb{C}(K_{1})\omin C_\mathbb{C}(K_{2})$ and the latter is again a $C^*$-algebra.

\begin{thm}
Let $K_{1}$ and $K_{2}$ be compact Hausdorff spaces, $f_{1},f_{2}\in\;C_\mathbb{C}(K_{1})$ and $g_{1},g_{2}\in\;C_\mathbb{C}(K_{2}).$ Then $f_{1}\otimes g_{1}\perp_{BJ}^{s}f_{2}\otimes g_{2}$ in $C_\mathbb{C}(K_{1})\otimes^{\min}C_\mathbb{C}(K_{2})$ if and only if either $f_{1}\perp_{BJ}^{s} f_{2}$ or $g_{1}\perp_{BJ}^{s} g_{2}.$
\end{thm}
\begin{proof} 
First suppose that $f_{1}\otimes g_{1}\perp_{BJ}^{s}f_{2}\otimes g_{2}.$ Using the identification as in Theorem \ref{idef1}, we have $f_1g_1 \perp_{BJ}^{s} f_2g_2$ in $C_\mathbb{C}(K_{1} \times K_{2})$. By \cite[Proposition 4.2]{H9}, there exists $(x_{0}, y_{0})\in K_1 \times K_2$ such that $|f_1(x_0)g_1(y_0)|= \|f_1\| \|g_1\|$ and $f_{2}(x_{0})g_{2}(y_{0})=0.$ The former condition implies that $||f_{1}||=|f_1(x_{0})|$ and $||g_{1}||=|g_{1}(x_{0})|$ and the latter one implies either $f_{2}(x_{0})=0$ or $g_{2}(y_{0})=0.$ 
 Thus, again by using  \cite[Proposition 4.2]{H9}, we have  $f_{1}\perp_{BJ}^{s}f_{2}$ or $g_{1}\perp_{BJ}^{s} g_{2}.$ Conversely, let us assume that $f_{1}\perp_{BJ}^{s} f_{2}.$ So there exists $x_{0}\in\;M_{f_{1}}$ such that $f_{2}(x_{0})=0.$  Since $M_{g_{1}} \neq \phi$, by fixing any $y_{0}\in\;M_{g_{1}}$, we have $f_{2}(x_{0})g_{2}(y_0)=0$ and $|f_1(x_0)g_1(y_0)|= \|f_1\| \|g_1\|$. Thus $f_{1}\otimes g_{1}\perp_{BJ}^{s}f_{2}\otimes g_{2}$.
\end{proof}
\begin{remark}
Above result gives that $f_{1}\otimes g_{1}\perp_{BJ}^{s}f_{2}\otimes g_{2}$ implies either $f_{1}\perp_{BJ}^{s} f_{2}$ or $g_{1}\perp_{BJ}^{s} g_{2}.$ So $f_{1}\bj  f_{2}h$ or $g_{1}\perp_{BJ} g_{2}k$ for every $h\in\;C_{\mathbb{C}}(K_{1})\;,\;k\in\;C_{\mathbb{C}}(K_{2}).$ Since strong BJ-orthogonality implies BJ-orthogonality, this gives $f_{1}\perp_{BJ}f_{2}$ or $g_{1}\perp_{BJ}g_{2}.$
\end{remark}
 
Lastly, we deviate our attention a little bit from the commutative setup to a non-commutative setting. It is quite obvious to ask what happens in the tensor product of non-commutative space $B(H)$, the space of bounded linear operators on a Hilbert space. In this direction, we obtain a positive result for the elementary tensors of rank one operators in $B(H)\otimes^{\min}B(H).$
\begin{propn}
Let $H$ be a Hilbert space and let $T_{1}, T_{2}, S_{1}, S_{2}\in\;B(H)$ be rank one operators satisfying $T_{1}\otimes T_{2}\perp_{BJ}S_{1}\otimes S_{2}$ in $B(H)\otimes^{\min}B(H).$ Then either $T_{1}\perp_{BJ}S_{1}$ or $T_{2}\perp_{BJ}S_{2}.$
\end{propn}
\begin{proof}
 We can embed  $B(H)\otimes^{\min}B(H)$ into $B(H\bar{\otimes} H)$ isometrically. Note that a rank one operator $T$ in $B(H)$ is of the form $x\bar{\otimes} y$ for some $x,y \in H$, where $(x\bar{\otimes} y)(h) = \langle h,y \rangle x$, for all $h\in H$.  It is easy to check that tensor product of two rank one operators is again a rank one operator on $H\bar{\otimes} H$, more precisely, if $T_{i}=x_{i}\bar{\otimes} y_{i}$, $i=1,2$, then $T_{1}\otimes T_{2}=(x_{1}\otimes x_{2})\bar{\otimes}(y_{1}\otimes y_{2}).$ Taking $S_{i}=z_{i}\bar{\otimes} w_{i}$,  $i=1,2$, since $T_{1}\otimes T_{2}\perp_{BJ}S_{1}\otimes S_{2},$ by \cite[Example 4.5]{H12}, either $\langle x_{1}\otimes x_{2}, z_{1}\otimes z_{2}\rangle=0$ or $\langle y_{1}\otimes y_{2}, w_{1}\otimes w_{2}\rangle=0$. That is, either $\langle x_{1}, z_{1}\rangle\langle x_{2}, z_{2}\rangle=0$ or $\langle y_{1}, w_{1}\rangle\langle y_{2}, w_{2}\rangle=0.$  The former case implies that either $\langle x_{1}, z_{1}\rangle=0$ or $\langle x_{2}, z_{2}\rangle=0,$ which gives either $T_{1}\perp_{BJ}S_{1}$ or $T_{2}\perp_{BJ}S_{2},$ using \cite[Example 4.5]{H12} again. Similarly one can check for the latter case. 
\end{proof}

\section*{acknowledgement}
The authors would like to thank Dr. Ved Prakash Gupta, School of Physical Sciences, Jawahar Lal Nehru University for his numerous valuable suggestions and discussions.

\end{document}